\documentclass[12pt,reqno]{amsart}

\usepackage{amsthm, amsmath, amsfonts, amssymb, graphicx, tikz-cd, calrsfs, hyperref}
\usepackage[shortlabels]{enumitem}
\usepackage[capitalize,noabbrev]{cleveref} 

\newtheorem{theorem}{Theorem}[section]
\newtheorem{lemma}[theorem]{Lemma}
\newtheorem{proposition}[theorem]{Proposition}
\newtheorem{corollary}[theorem]{Corollary}

\theoremstyle{definition}
\newtheorem{definition}[theorem]{Definition}

\newtheorem{remark}[theorem]{Remark}

\newcounter{mycount}

\newcommand{\myref}[1]{\hyperref[#1]{#1}}

\def\Frm{{\sf{Frm}}}
\def\Loc{{\sf{Loc}}}
\def\LPries{{\sf{LPries}}}

\def\up{{\uparrow}}
\def\down{{\downarrow}}

\def\pf{{\sf{pf}}}

\def\ClopUp{{\sf{ClopUp}}}
\newcommand{\cl}{{\sf{cl}}}
\newcommand{\Int}{{\sf{int}}}

\title{A New Proof of the Joyal-Tierney Theorem}

\author{G. Bezhanishvili}
\address{New Mexico State University}
\email{guram@nmsu.edu}

\author{L. Carai}
\address{Universitat de Barcelona}
\email{luca.carai.uni@gmail.com}

\author{P. Morandi}
\address{New Mexico State University}
\email{pmorandi@nmsu.edu}

\begin{document}

\begin{abstract}
We give an alternative, more geometric, proof of the well-known Joyal-Tierney Theorem in locale theory by utilizing Priestley duality for frames.
\end{abstract}



\maketitle

\section{Introduction and Preliminaries}

A well-known result in locale theory, known as the Joyal-Tierney Theorem, states that a localic map $f : L \to M$ is open iff its left adjoint $f^*$ is a complete Heyting homomorphism (see, e.g., \cite[Prop.~III.7.2]{PP12}). 
In addition, if $L$ is subfit, then $f$ is open iff $f^*$ is a complete lattice homomorphism 
(see, e.g., \cite[Prop.~V.1.8]{PP12}). Our aim is to give another, more geometric, proof of this result utilizing the language of Priestley spaces.

Priestley duality \cite{Pri70,Pri72} establishes a dual equivalence between the categories of bounded distributive lattices and Priestley spaces. We recall that a {\em Priestley space} is a Stone space $X$ equipped with a partial order $\le$ such that 
$x\not\le y$ implies the existence of a clopen upset $U$ such that $x\in U$ and $y\notin U$. A \emph{Priestley morphism} is a continuous order-preserving map. 

Pultr and Sichler \cite{PS88} showed how to restrict Priestley duality to the category of frames. We recall (see, e.g., \cite[p.~10]{PP12}) that a {\em frame} is a complete lattice $L$ satisfying the infinite distributive law 
$
a\wedge\bigvee S = \bigvee \{a\wedge s : s \in S \}
$
for each $a \in L$ and $S\subseteq L$. A map $h : L\to M$ between frames is a {\em frame homomorphism} if $h$ preserves finite meets and arbitrary joins. Let $\Frm$ be the category of frames and frame homomorphisms. 

\begin{definition}
\hfill
\begin{enumerate}
\item A Priestley space $X$ is a {\em localic space}, or simply an \emph{$L$-space}, provided the closure of an open upset is a clopen upset. 
\item A Priestley morphism $f : X \to Y$ between $L$-spaces is an {\em $L$-morphism} provided $\cl f^{-1}U = f^{-1} \cl U$ for each open upset $U$ of $Y$.
\item Let $\LPries$ be the category of $L$-spaces and $L$-morphisms.
\end{enumerate}
\end{definition}

\begin{proposition} \cite[p.~198]{PS88} \label{prop: PS duality}
$\Frm$ is dually equivalent to $\LPries$.
\end{proposition}

\begin{remark}
Since frames are exactly complete Heyting algebras (see, e.g., \cite[Prop.~1.5.4]{Esa19}), 
every $L$-space is an Esakia space, where we recall that a Priestley space $X$ is an {\em Esakia space} provided $\down U$ is clopen for each clopen $U \subseteq X$. 
\end{remark}

\begin{remark}
The contravariant functors establishing Pultr-Sichler duality 
are the restrictions of the contravariant functors establishing Priestley duality. They are described as follows. 

For an $L$-space $X$, let $\ClopUp(X)$ be the frame of clopen upsets of $X$. The 
functor $\ClopUp:\LPries \to\Frm$ sends $X \in \LPries$ to $\ClopUp(X)$ and an $\LPries$-morphism $f :X \to Y$ to the $\Frm$-morphism $f^{-1} : \ClopUp(Y) \to \ClopUp(X)$. 

For $L \in \Frm$ let $X_L$ be the set of prime filters of $L$ ordered by inclusion and equipped with the topology whose basis is $\{ \phi(a)\setminus\phi(b) : a, b \in L \}$, where $\phi : L \to \wp(X_L)$ is the {\em Stone map} $\phi(a) = \{ x \in X_L : a \in x\}$. Then $X_L$ is an $L$-space and the 
functor $\pf : \Frm \to \LPries$ sends $L \in \Frm$ to $X_L$ and a $\Frm$-morphism $h : L \to M$ to $h^{-1} : X_M \to X_L$.
\end{remark}

Let $L,M$ be frames. Every frame homomorphism $h:L\to M$ has a right adjoint $r=h_*:M\to L$, called a {\em localic map}. It is given by $r(b)=\bigvee\{a\in L : h(a)\le b \}$. 
The following  provides a characterization of localic maps:

\begin{proposition} \cite[Prop.~II.2.3]{PP12} \label{prop: localic map}
A map $r : M \to L$ between frames is a localic map iff
\begin{enumerate}[$(1)$]
\item $r$ preserves all meets $($so has a left adjoint $h=r^*);$ 
\item $r(a) = 1$ implies $a = 1;$
\item $r(h(a) \to b) = a \to r(b).$ 
\end{enumerate}
\end{proposition}

Let $\Loc$ be the category of frames and localic maps. The following is obvious from \cref{prop: PS duality,prop: localic map}:

\begin{proposition}
$\Loc$ is dually isomorphic to $\Frm$, and hence equivalent to $\LPries$.
\end{proposition}

To define open localic maps, we recall the notion of a sublocale which generalizes that of a subspace. Let $L$ be a frame. A subset $S$ of $L$ is a \emph{sublocale} of $L$ if $S$ is closed under arbitrary meets and $x \to s \in S$ for each $x \in L$ and $s \in S$. Sublocales correspond to nuclei, where we recall (see, e.g., \cite[Sec.~III.5.3]{PP12}) that a {\em nucleus} on $L$ is a map $\nu : L \to L$ satisfying
\begin{enumerate}
\item $a \le \nu a$;
\item $\nu\nu a \le \nu a$;
\item $\nu(a\wedge b) = \nu a \wedge \nu b$.
\end{enumerate}

We can go back and forth between nuclei and sublocales as follows. If $\nu$ is a nucleus on $L$, then $S_\nu := \nu[L]$ is a sublocale of $L$. Conversely, if $S$ is a sublocale of $L$, then $\nu_S : L \to L$ is a nucleus on $L$, where $\nu_S(a) = \bigwedge \{ s \in S : a \le s\}$. This correspondence is one-to-one (see, e.g., \cite[Prop.~III.5.3.2]{PP12}). 

If $a \in L$, then $\mathfrak o(a) := \{a \to x : x \in L\}$ is a sublocale of $L$, called an \emph{open sublocale} of $L$, whose corresponding nucleus $\nu_a$ is given by $\nu_a(x) = a \to x$ (see, e.g., \cite[pp.~33, 35]{PP12}).

\begin{definition} \cite[p.~37]{PP12}
A localic map $r : M \to L$ is {\em open} if for each open sublocale $S$ of $M$, the image $r[S]$ is an open sublocale of $L$.
\end{definition}

\section{The Joyal-Tierney Theorem}

The Joyal-Tierney Theorem provides a characterization of open localic maps (see, e.g., \cite[Prop.~7.3]{PP08} or \cite[pp.~37--38]{PP12}): 

\begin{theorem} [Joyal-Tierney] \label{thm: JT}
Let $r : M \to L$ be a localic map between frames with left adjoint $h$. The following are equivalent:
\begin{enumerate}[$(1)$]
\item $r$ is open.
\item $h$ is a complete Heyting homomorphism.
\item $h$ has a left adjoint $\ell=h^*$ satisfying the Frobenius condition $\ell(a \wedge h(b)) = \ell(a) \wedge b$ for each $a \in M$ and $b \in L$.
\end{enumerate}
\end{theorem}

Our aim is to give an alternative proof of this result using Priestley duality for frames. 
For this we need to translate the algebraic conditions of \cref{thm: JT} into geometric conditions about Priestley spaces.
We will freely use the following well-known lemma.
For parts (1) and (2) see \cite[Lems.~11.21, 11.22]{DP02}; for part (3) see \cite[Prop.~2.6]{Pri84}; and part (4) is a consequence of Esakia's lemma (see \cite[Lem.~3.3.12]{Esa19}). 

\begin{lemma}\label{lem: facts Priestley}
\hfill
\begin{enumerate}[label=$(\arabic*)$, ref=\thelemma(\arabic*)]
\item \label[lemma]{lem: facts Priestley 1} For a Priestley space $X$, the set $\{ U \setminus V : U,V \in \ClopUp(X) \}$ is a basis of open sets of $X$.
\item \label[lemma]{lem: facts Priestley 2} Let $X$ be a Priestley space. If $F,G$ are disjoint closed subsets of $X$, with $F$ an upset and $G$ a downset, then there exist disjoint clopen subsets $U,V$ of $X$, with $U$ an upset and $V$ a downset, such that $F \subseteq U$ and $G \subseteq V$. In particular, every open upset is a union and every closed upset is an intersection of clopen upsets.
\item \label[lemma]{lem: facts Priestley 3}If $F$ is a closed subset of a Priestley space, then $\up F$ and $\down F$ are closed.
\item \label[lemma]{lem: facts Priestley 4} Let $f:X \to Y$ be a continuous map between Priestley spaces. For each $x \in X$ we have 
\[
f\left[ \bigcap \{ U \in \ClopUp(X) : x \in U\} \right] = \bigcap \{ f[U] : x \in U \in \ClopUp(X) \}.
\]
\end{enumerate}
\end{lemma}

We recall (see, e.g., \cite[p.~265]{DP02}) that if $h : L \to M$ is a frame homomorphism and $f : X_M \to X_L$ is its Priestley dual, then
\begin{equation}
f^{-1}(\phi(a)) = \phi(h(a)) \label{**} \tag{i}. 
\end{equation}
We also recall that if $r:M\to L$ is a localic map and $S$ is a sublocale of $M$, then $r[S]$ is a sublocale of $L$ (see, e.g., \cite[Prop.~III.4.1.]{PP12}). 

\begin{lemma} \label{lem: nuclei corresponding to the image of a sublocale}
Let $r : M \to L$ be a localic map with left adjoint $h$. If $S$ is a sublocale of $M$, then 
$\nu_{r[S]} = r\nu_S h$.
\end{lemma}

\begin{proof}
Let $a \in L$. We have
\begin{align*}
\nu_{r[S]}(a) &= 
\bigwedge \{ r(s) : s \in S, \, a \le r(s)\} \\
&= \bigwedge \{ r(s) : s \in S, \, h(a) \le s\} \\
&= r\left(\bigwedge \{ s \in S : h(a) \le s\}\right) \\
&= r\nu_Sh(a). 
\end{align*}
Therefore, $\nu_{r[S]} = r\nu_S h$.
\end{proof}

We thus see that a localic map $r : M \to L$, with left adjoint $h$, is open iff for each $a \in M$ there is $b \in L$ with $r\nu_a h = \nu_b$. We use this observation in the proof of the following lemma.

\begin{lemma} \label{lem: translation of r open}
Let $r : M \to L$ be a localic map, $h$ the left adjoint of $r$, and $f : X_M \to X_L$ the Priestley dual of $h$. The following are equivalent:
\begin{enumerate}[$(1)$]
\item $r$ is open.
\item If $U$ is a clopen upset of $X_M$, then $f[U]$ is a clopen upset of $X_L$.
\end{enumerate}
\end{lemma}

\begin{proof}
We start by showing that if $a \in M$ and $b, c \in L$, then
\begin{equation}\tag{ii}
b \le (r\nu_a h)(c) \Longleftrightarrow \phi(b) \cap f[\phi(a)] \subseteq \phi(c).  \label{*}
\end{equation}
To see this,
\begin{align*}
b \le (r\nu_a h)(c) &\Longleftrightarrow b \le r(a \to h(c)) \Longleftrightarrow h(b) \le a \to h(c) \\
&\Longleftrightarrow h(b) \wedge a \le h(c).
\end{align*}
Therefore, since $f[f^{-1}(B) \cap A] = B \cap f[A]$ for each $A, B$, by (\ref{**}) we have
\begin{align*}
b \le (r\nu_a h)(c) &\Longleftrightarrow \phi(h(b)) \cap \phi(a) \subseteq \phi(h(c)) \Longleftrightarrow f^{-1}(\phi(b)) \cap \phi(a) \subseteq f^{-1}(\phi(c)) \\
&\Longleftrightarrow f[f^{-1}(\phi(b)) \cap \phi(a)] \subseteq \phi(c) \Longleftrightarrow \phi(b) \cap f[\phi(a)] \subseteq \phi(c).
\end{align*}

(1)$\Rightarrow$(2). Let $U \in \ClopUp(X_M)$. Then $U = \phi(a)$ for some $a \in M$. By (1) and \cref{lem: nuclei corresponding to the image of a sublocale}, there is $b \in L$ with $r\nu_a h = \nu_b$. Since $1 = \nu_b(b)$, we have $1 \le (r\nu_a h)(b)$, so $\phi(1) \cap f[U] \subseteq \phi(b)$ by (\ref{*}). Therefore, $f[U] \subseteq \phi(b)$. For the reverse inclusion, let $y \in \phi(b)$. If $y \notin f[U]$, then since $f[U]$ is closed in $X_L$, there is a clopen set containing $y$ and missing $f[U]$. By \cref{lem: facts Priestley 1}, there are $c, d \in L $ with $y \in \phi(c) \setminus \phi(d)$ and $f[U] \cap (\phi(c) \setminus \phi(d)) = \varnothing$. Thus, $f[U] \cap \phi(c) \subseteq \phi(d)$, so $c \le (r\nu_a h)(d) = \nu_b(d) = b \to d$ by (\ref{*}). This gives $b \wedge c \le d$, and hence $\phi(b) \cap \phi(c) \subseteq \phi(d)$, a contradiction since $y \in \phi(b) \cap \phi(c)$ but $y \notin \phi(d)$. Therefore, $y \in f[U]$, and so $\phi(b) \subseteq f[U]$. Consequently, $f[U] = \phi(b)$, and so $f[U] \in \ClopUp(X_L)$.

(2)$\Rightarrow$(1). Let $a \in M$ and set $U = \phi(a)$. Then $U\in\ClopUp(X_M)$, so $f[U]\in\ClopUp(X_L)$ by (2). Therefore, there is $b \in L$ with $\phi(b) = f[U]$. If $c, d \in L$, then by (\ref{*}), 
\begin{align*}
c \le (r\nu_a h)(d) &\Longleftrightarrow \phi(c) \cap f[U] \subseteq \phi(d) \\
&\Longleftrightarrow \phi(c) \cap \phi(b) \subseteq \phi(d) \\
&\Longleftrightarrow c \wedge b \le d \\
&\Longleftrightarrow c \le b \to d \\
&\Longleftrightarrow c \le \nu_b(d). 
\end{align*}
Thus, $r\nu_a h = \nu_b$, and hence $r$ is open.
\end{proof}

We next give a dual characterization of when a frame homomorphism has a left adjoint.
Let $X$ be a Priestley space. Then we have two additional topologies on $X$, the topology of open upsets and the topology of open downsets. If $\Int_i$ and $\cl_i$ are the corresponding interior and closure operators ($i=1,2$), then it is well known (see, e.g., \cite[Lem.~6.5]{BBGK10}) that for $A \subseteq X$ we have:
\begin{eqnarray*}
\Int_1(A) = X \setminus \down (X \setminus\Int A ) & \mbox{ and } & \cl_1 A = \down\cl A; \\
\Int_2(A) = X \setminus \up (X \setminus\Int A ) & \mbox{ and } & \cl_2 A = \up\cl A.
\end{eqnarray*}

Let $L$ be a frame and let $a=\bigwedge S$ for $a\in L$ and $S\subseteq L$. Then $\phi(a) = \Int_1 \bigcap \{ \phi(s) : s \in S\}$ (see, e.g., \cite[Lem.~2.3]{BB08}). This will be used in the following lemma.

\begin{lemma} \label{lem: translation of h complete}
Let $h : L \to M$ be a frame homomorphism and $f: X_M \to X_L$ its Priestley dual. The following are equivalent:
\begin{enumerate}[$(1)$]
\item $h$ has a left adjoint.
\item $h$ preserves all meets.
\item $f^{-1}\Int_1 F = \Int_1 f^{-1}F$ for each closed upset $F \subseteq X_L$.
\item $\up f[U]$ is clopen for each clopen upset $U \subseteq X_M$.
\end{enumerate}
\end{lemma}

\begin{proof}
(1)$\Leftrightarrow$(2). This is well known (see, e.g., \cite[Prop.~7.34]{DP02}).

(2)$\Rightarrow$(3). Let $F$ be a closed upset of $X_L$. We may write $F = \bigcap \{ \phi(s) : s \in S\}$ for some $S \subseteq L$ (see \cref{lem: facts Priestley 2}). 
By (\ref{**}), 
\[
f^{-1}(F) = f^{-1}\left( \bigcap \{ \phi(s) : s \in S\} \right) = \bigcap \{ f^{-1}(\phi(s)) : s \in S\} = \bigcap \{ \phi(h(s)) : s \in S\},
\]
so
\[
\Int_1 f^{-1}(F) = \Int_1 \bigcap \{ \phi(h(s)) : s \in S\} = \phi\left(\bigwedge h[S]\right). 
\]
On the other hand, since
\[
\Int_1 F = \Int_1 \bigcap \{ \phi(s) : s \in S\} = \phi\left(\bigwedge S\right), 
\]
using (\ref{**}) again yields
\[
f^{-1}(\Int_1 F) = f^{-1}\left(\phi\left(\bigwedge S\right)\right) = \phi\left(h\left(\bigwedge S\right)\right).
\]
Therefore, 
by (2) we have
\begin{align*}
\Int_1 f^{-1}(F) = \phi\left(\bigwedge h[S]\right) = \phi\left(h\left(\bigwedge S\right)\right) = f^{-1}(\Int_1 F).
\end{align*}

(3)$\Rightarrow$(4). Let $U \in \ClopUp(X_M)$ and set $F = \up f[U]$. By \cref{lem: facts Priestley 3}, $F$ is a closed upset of $Y$. By~(3),
\[
U \subseteq \Int_1 f^{-1}(f[U]) \subseteq \Int_1 f^{-1}F = f^{-1}\Int_1 F,
\]
so $f[U] \subseteq \Int_1 F$, and hence $\up f[U] \subseteq \Int_1 F = \Int_1 \up f[U]$. Thus, $\up f[U]$ is clopen.

(4)$\Rightarrow$(1). Let $a \in M$. By (4), $\up f[\phi(a)] \in \ClopUp(X_L)$. Therefore, there is a unique $b \in L$ such that $\phi(b)=\up f[\phi(a)]$. Letting $\ell(a)=b$ defines a function $\ell : M \to L$ such that
\begin{equation}
\phi(\ell(a))=\up f[\phi(a)]. \label{***} \tag{iii}
\end{equation}
To see that $\ell$ is left adjoint to $h$, let $c \in L$. Since $\phi(c)$ is an upset, by (\ref{**}) we have
\begin{align*}
\ell(a) \le c &\Longleftrightarrow \phi(\ell(a)) \subseteq \phi(c) \Longleftrightarrow \up f[\phi(a)] \subseteq \phi(c) \Longleftrightarrow f[\phi(a)] \subseteq \phi(c)\\
&\Longleftrightarrow \phi(a) \subseteq f^{-1}(\phi(c)) \Longleftrightarrow \phi(a) \subseteq \phi(h(c)) \Longleftrightarrow a \le h(c). \qedhere
\end{align*}
\end{proof}

We recall (see, e.g., \cite[p.~9]{Esa19}) that a map $f:X\to Y$ between posets is a {\em bounded morphism} or a {\em p-morphism} if 
$\down f^{-1}(y)=f^{-1}(\down y)$ for each $y\in Y$. Let $h:L\to M$ be a frame homomorphism between frames and $f:X_M\to X_L$ its Priestley dual. Then $f$ is an $L$-morphism. It follows from Esakia duality for Heyting algebras \cite{Esa74, Esa19} that $h$ preserves $\to$ iff $f$ is a p-morphism. This together with \cref{lem: translation of h complete} yields:

\begin{lemma} \label{lem: CHA dually}
Let $h : L \to M$ be a frame homomorphism with Priestley dual $f:X_M\to X_L$. Then $h$ is a complete Heyting homomorphism iff $f$ is a p-morphism and $\up f[U]$ is clopen for each clopen upset $U$ of $X_M$.
\end{lemma}

We next provide a dual characterization of the Frobenius condition $\ell(a \wedge h(b)) = \ell(a) \wedge b$ for each $a \in M$ and $b \in L$. 

\begin{lemma} \label{lem: translation of Frobenius}
Let $h : L \to M$ be a frame homomorphism with Priestley dual $f:X_M\to X_L$. The following are equivalent:
\begin{enumerate}[$(1)$]
\item $h$ has a left adjoint $\ell$ and $\ell(a \wedge h(b)) = \ell(a) \wedge b$ for each $a \in M$ and $b \in L$.
\item $\up f[U]$ is clopen and $\up(f[U] \cap V) = \up f[U] \cap V$ for each $U \in \ClopUp(X_M)$ and $V \in \ClopUp(X_L)$.
\end{enumerate}
\end{lemma}

\begin{proof}
By \cref{lem: translation of h complete}, $h$ has a left adjoint $\ell$ iff $\up f[U]$ is clopen for each $U \in \ClopUp(X_M)$. It is left to show that $\ell(a \wedge h(b)) = \ell(a) \wedge b$ for each $a \in M$ and $b \in L$ iff $\up(f[U] \cap V) = \up f[U] \cap V$ for each $U \in \ClopUp(X_M)$ and $V \in \ClopUp(X_L)$. Letting $U = \phi(b)$ and $V = \phi(a)$, since $\up f[U] = \phi(\ell(a))$ by (\ref{***}), we have
\[
\phi(\ell(a) \wedge b) = \phi(\ell(a)) \cap \phi(b) = \up f[U] \cap V.
\]
On the other hand, since $f[U \cap f^{-1}(V)] = f[U] \cap V$, by (\ref{**}) we have
\begin{align*}
\phi(\ell(a \wedge h(b))) &= \up f[\phi(a \wedge h(b))] = \up f[\phi(a) \cap f^{-1}(\phi(b))] \\
& = \up f[U \cap f^{-1}(V)] = \up(f[U] \cap V).
\end{align*}
Thus, 
\begin{align*}
\ell(a \wedge h(b)) = \ell(a) \wedge b & \Longleftrightarrow \phi(\ell(a \wedge h(b))) = \phi(\ell(a) \wedge b) \\
& \Longleftrightarrow \up(f[U] \cap V) = \up f[U] \cap V. \qedhere
\end{align*}
\end{proof}

We thus have translated the three conditions of \cref{thm: JT} into the dual conditions in the language of Priestley spaces. We next prove that the translated conditions are equivalent.

\begin{theorem} \label{thm: JT in Priestley terms}
Let $f : X \to Y$ be a Priestley morphism between $L$-spaces. 
The following are equivalent:
\begin{enumerate}[$(1)$]
\item If $U\in\ClopUp(X)$, then $f[U]\in\ClopUp(Y)$.
\item $f$ is a p-morphism and $\up f[U]$ is clopen for each $U\in\ClopUp(X)$.
\item $\up f[U]$ is clopen and $\up(f[U] \cap V) = \up f[U] \cap V$ for each $U \in \ClopUp(X)$ and $V \in \ClopUp(Y)$.
\end{enumerate}
\end{theorem}

\begin{proof}
(1)$\Rightarrow$(2). Let $U \in \ClopUp(X)$. By (1), $f[U]$ is an upset of $Y$, so $\up f[U] = f[U]$. Therefore, $\up f[U]$ is clopen in $Y$ by (1). It is left to prove that $f$ is a p-morphism. For this it suffices to show that $f(\up x)$ is an upset for each $x\in X$ (see, e.g, \cite[Prop~1.4.12]{Esa19}). By \cref{lem: facts Priestley 2}, $\up x = \bigcap \{ U \in \ClopUp(X) : x \in U\}$, so by \cref{lem: facts Priestley 4}, 
\[
f[\up x] = f\left[ \bigcap \{ U \in \ClopUp(X) : x \in U\} \right] = \bigcap \{ f[U] : x \in U \in \ClopUp(X) \}.
\] 
Thus, $f[\up x]$ is an upset by (1).

(2)$\Rightarrow$(3). It is sufficient to show that $\up(f[U] \cap V) = \up f[U] \cap V$ for each $U \in \ClopUp(X)$ and $V \in \ClopUp(Y)$. But since $f$ is a p-morphism, $\up f[U] = f[U]$, so $\up f[U] \cap V = f[U] \cap V = \up(f[U] \cap V)$ because $f[U] \cap V$ is an upset.

(3)$\Rightarrow$(1). It suffices to show that $f[U]$ is an upset. If not, then there exist $x \in U$ and $y \in Y$ with $f(x) \le y$ but $y \notin f[U]$. This yields $y \notin \down (\down y \cap f[U])$, so there is a clopen upset $V$ of $Y$ such that $y \in V$ and $V \cap \down y \cap f[U] = \varnothing$ (see \cref{lem: facts Priestley 2}). Therefore, $y \notin \up (f[U] \cap V)$ but $y \in \up f[U] \cap V$, 
a contradiction to (3). Thus, $f[U]$ is an upset.
\end{proof}

By \cref{lem: translation of r open,lem: CHA dually,lem: translation of Frobenius}, the three conditions of \cref{thm: JT in Priestley terms} are equivalent to the corresponding three conditions of \cref{thm: JT}. Hence, the Joyal-Tierney Theorem is a consequence of \cref{thm: JT in Priestley terms}. We conclude this section with the following observation.

\begin{remark}
Condition (1) of \cref{thm: JT in Priestley terms} is equivalent to: 
\begin{enumerate}
\item[(1$^\prime$)] If $U$ is an open upset of $X$, then $f[U]$ is an open upset of $Y$.
\end{enumerate}
Clearly (1$^\prime$) implies (1) since if $U$ is clopen, then $f[U]$ is closed, hence a clopen upset of $Y$ by (1$^\prime$). Conversely, if $U$ is an open upset, then $U = \bigcup \{ V \in \ClopUp(X) : V \subseteq U\}$ by \cref{lem: facts Priestley 2}. Therefore, $f[U] = \bigcup \{ f[V] : V \in \ClopUp(X), \ V \subseteq U\}$ is a union of clopen upsets of $Y$ by (1). Thus, $f[U]$ is an open upset of $Y$. Consequently, (1) is equivalent to $f$ being an open map with respect to the open upset topologies. 

On the other hand, this does not imply that $f$ is an open map with respect to the Stone topologies. To see this, we use the space defined in \cite[p.~32]{Bez99}.  Let $X$ be the 2-point compactification of the discrete space $\{ x_n, z_n : n \in \mathbb{N}\}$ with $\omega$ the limit point of $\{x_n : n \in \mathbb{N}\}$ and $\omega'$ the limit point of $\{ z_n : n \in \mathbb{N} \}$. Also, let $Y$ be the 1-point compactification of the discrete space $\{ y_n : n \in \mathbb{N} \}$. We order $X$ and $Y$ and define the map $f : X \to Y$ as shown in the diagram below. 
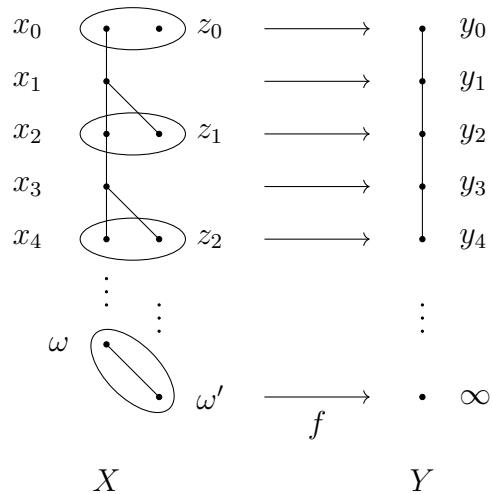
\begin{center}
\begin{tikzpicture}[scale=.7]
\newcommand{\sep}{6}
\newcommand{\pre}{1}
\newcommand{\init}{1}
\node [above] at (0,-2) {$X$};
\draw [fill] (1,0) circle[radius=.05];
\draw [fill] (0,1) circle[radius=.05];
\draw [fill] (0,1.75) circle[radius=.02];
\draw [fill] (0,2) circle[radius=.02];
\draw [fill] (0,2.25) circle[radius=.02];
\draw [fill] (0,3) circle[radius=.05];
\draw [fill] (1,3) circle[radius=.05];
\draw [fill] (0,4) circle[radius=.05];
\draw [fill] (0,5) circle[radius=.05];
\draw [fill] (1,5) circle[radius=.05];
\draw [fill] (0,6) circle[radius=.05];
\draw [fill] (0,7) circle[radius=.05];
\draw [fill] (1,7) circle[radius=.05];
\draw (1,0) -- (0,1);
\draw (0,3) -- (0,7);
\draw (1,3) -- (0,4);
\draw (1,5) -- (0,6);
\node [right] at (1.5,0) {$\omega'$};
\node [left] at (-.5,1) {$\omega$};
\node[right] at (1.5,3) {$z_2$};
\node [left] at (-\pre,3) {$x_4$};
\node [left] at (-\pre,4) {$x_3$};
\node [left] at (-\pre,5) {$x_2$};
\node [left] at (-\pre,6) {$x_1$};
\node [left] at (-\pre,7) {$x_0$};
\node [right] at (1.5,5) {$z_1$};
\node [right] at (1.5,7) {$z_0$};
\draw [fill] (1,1.25) circle[radius=.02];
\draw [fill] (1,1.5) circle[radius=.02];
\draw [fill] (1,1.75) circle[radius=.02];
\draw [rotate=45] (.70,0) ellipse (0.5 and 1.0);
\draw (.5,3) ellipse (1.0 and 0.4);
\draw (.5,5) ellipse (1.0 and 0.4);
\draw (.5,7) ellipse (1.0 and 0.4);
\node [above] at (\sep, -2) {$Y$};
 \draw [fill] (\sep,0) circle[radius=.05];
 \draw [fill] (\sep,1.25) circle[radius=.02];
\draw [fill] (\sep,1.5) circle[radius=.02];
\draw [fill] (\sep,1.75) circle[radius=.02];
\draw [fill] (\sep,3) circle[radius=.05]; 
\draw [fill] (\sep,4) circle[radius=.05];
\draw [fill] (\sep,5) circle[radius=.05];
\draw [fill] (\sep,6) circle[radius=.05];
\draw [fill] (\sep,7) circle[radius=.05];
\node [right] at (.5+\sep,0) {$\infty$};
\node [right] at (.5+\sep, 3) {$y_4$};
\node [right] at (.5+\sep, 4) {$y_3$};
\node [right] at (.5+\sep, 5) {$y_2$};
\node [right] at (.5+\sep, 6) {$y_1$};
\node [right] at (.5+\sep, 7) {$y_0$};
\draw (\sep, 3) -- (\sep, 7);
\node [above] at (4,-1) {$f$};
\draw [->] (2 + \init,0) -- (\sep - \pre,0);
\draw [->] (2 + \init,3) -- (\sep - \pre,3);
\draw [->] (2 + \init,4) -- (\sep - \pre,4);
\draw [->] (2 + \init,5) -- (\sep - \pre,5);
\draw [->] (2 + \init,6) -- (\sep - \pre,6);
\draw [->] (2 + \init,7) -- (\sep - \pre,7);
\end{tikzpicture}
\end{center}
It is straightforward to see that $X$ and $Y$ are $L$-spaces and $f$ is an $L$-morphism such that $f[U]$ is a clopen upset of $Y$ for each clopen upset $U$ of $X$. However, $f$ is not an open map since $U := \{ z_n : n \in \mathbb{N}\} \cup \{\omega'\}$ is an open subset of $X$ whose image $\{ y_{2n} : n \in \mathbb{N} \} \cup \{\infty\}$ is not an open subset of $Y$.
\end{remark}

\section{The subfit case}

As was shown in \cite[Prop.~V.1.8]{PP12}, if in the Joyal-Tierney Theorem we assume that $L$ is subfit, then the localic map $r:M\to L$ is open iff its left adjoint $h:L\to M$ is a complete lattice homomorphism (so $h$ being a Heyting homomorphism becomes redundant). We will give an alternative proof of this result in the language of Priestley spaces. 

We recall that a frame $L$ is {\em subfit} if for all $a,b\in L$ we have
\[
a\not\le b \Longrightarrow (\exists c \in L)(a\vee c = 1 \mbox{ but } b\vee c\ne 1).
\]
We next give a dual characterization of when $L$ is subfit. As usual, for a poset $X$ we write $\min X$ for the set of minimal elements of $X$.

\begin{lemma}
Let $L$ be a frame and $X_L$ its Priestley space. Then $L$ is subfit iff $\min X_L$ is dense in $X_L$.
\end{lemma}

\begin{proof}
First suppose that $\min X_L$ is dense in $X_L$. Let $a, b \in L$ with $a \not\le b$. Then $\phi(a) \not\subseteq \phi(b)$, so $\phi(a) \setminus \phi(b)$ is a nonempty clopen subset of $X$. Therefore, there is $x \in (\phi(a) \setminus \phi(b)) \cap \min X_L$. Let $U = X_L \setminus \{x\}$. Then $U$ is an open upset of $X_L$. Since $\phi(a) \cup U = X_L$ and $U$ is a union of clopen upsets (see \cref{lem: facts Priestley 2}), compactness of $X_L$ implies that there is a clopen upset $U' \subseteq U$ with $\phi(a) \cup U' = X_L$. Because $U'=\phi(c)$ for some $c \in L$, we have $a \vee c = 1$. On the other hand, since $\phi(b) \cup U' \ne X_L$, it follows that $b \vee c \ne 1$. Thus, $L$ is subfit.

Conversely, suppose that $\min X_L$ is not dense in $X_L$. Then there is a nonempty clopen subset $A$ of $X_L$ such that $A \cap \min X_L = \varnothing$. We may assume that $A = U \setminus V$, where $U \not\subseteq V$ are clopen upsets of $X_L$ (see \cref{lem: facts Priestley 1}). From $A \cap \min X_L = \varnothing$ it follows that $U \cap \min X_L \subseteq V$. Let $a, b \in L$ be such that $U = \phi(a)$ and $V = \phi(b)$. Since $U\not\subseteq V$, we have $a\not\le b$. Suppose $c \in L$ is such that $a \vee c = 1$. 
Let $W = \phi(c)$. Then $U \cup W = X_L$, so $\min X_L \subseteq U \cup W$. Because $U \cap \min X_L \subseteq V$, this yields $\min X_L \subseteq V \cup W$, which forces $V \cup W = X_L$ because $\up \min X_L = X_L$ (see, e.g.,  \cite[Thm.~3.2.1]{Esa19}). Thus, $b \vee c = 1$, and hence $L$ is not subfit.
\end{proof}

\begin{lemma} \label{lem: subfit implies p-morphism}
Let $f : X \to Y$ be a Priestley morphism between $L$-spaces. 
If $\min Y$ is dense in $Y$ and $\up f[U]$ is clopen for each $U \in \ClopUp(X)$,
then $f$ is a p-morphism.
\end{lemma}

\begin{proof}
By the same argument as in the proof of (1)$\Rightarrow$(2) of \cref{thm: JT in Priestley terms}, it suffices to show that $f[U]$ is an upset for each $U \in \ClopUp(X)$. If not,
then $\up f[U]\setminus f[U]\ne\varnothing$ for some $U\in\ClopUp(X)$. Let $V = \up f[U]\setminus f[U]$. Since $f[U]$ is closed, $\up f[U]$ is closed by \cref{lem: facts Priestley 3}. Therefore, $V$ is a nonempty open subset of $Y$. Thus, $V \cap \min Y \ne \varnothing$. On the other hand,
\[
V \cap \min Y \subseteq \up f[U] \cap \min Y \subseteq f[U] \cap \min Y.
\]
This is a contradiction since $V \cap f[U] = \varnothing$. Consequently, $f[U]$ is an upset.
\end{proof}

As an immediate consequence of \cref{lem: subfit implies p-morphism}, we obtain:

\begin{theorem} \label{thm: PP version of JT}
Let $f : X \to Y$ be a Priestley morphism between $L$-spaces. 
If $\min Y$ is dense in $Y$, then Condition $(2)$ in Theorem~\emph{\ref{thm: JT in Priestley terms}} is equivalent to
\begin{enumerate}
\item[$(2^\prime)$] $\up f[U]$ is clopen for each $U\in\ClopUp(X)$.
\end{enumerate}
\end{theorem}

\cref{thm: JT in Priestley terms,thm: PP version of JT} together with \cref{lem: translation of r open,lem: translation of h complete}
 yield the following version of the Joyal-Tierney Theorem:

\begin{corollary} \cite[Prop.~V.1.8]{PP12}
Let $r : M \to L$ be a localic map with left adjoint $h$. If $L$ is subfit, then $r$ is open iff $h$ is a complete lattice homomorphism.
\end{corollary}

\end{document}